\documentclass[10pt,psamsfonts]{amsart}
\usepackage{amsmath}
\usepackage{amsthm}
\usepackage{amssymb}
\usepackage{amscd}
\usepackage{amsfonts}
\usepackage{amsbsy}
\usepackage{color}
\usepackage{graphicx,psfrag}
\usepackage{graphics}

\newcommand{\R}{\ensuremath{\mathbb{R}}}

\newcommand{\C}{\ensuremath{\mathbb{C}}}

\newcommand{\sss}{\ensuremath{\mathbb{S}}}

\newcommand{\X}{\mathcal{X}}
\newcommand{\Y}{\mathcal{Y}}

\def\p{\partial}

\newtheorem {theorem} {Theorem}
\newtheorem {proposition} [theorem]{Proposition}

\newtheorem {remark} [theorem]{Remark}

\begin{document}

\title[Darboux theory of integrability on $\sss^n$]
{Darboux theory of integrability for real polynomial vector fields
on $\sss^n$}
\author[J. Llibre and A.C. Murza]
{Jaume Llibre and Adrian C. Murza}
\address{Jaume Llibre, Departament de Matem\`atiques, Universitat Aut\`onoma de Barcelona,
08193 Bellaterra, Barcelona, Spain}
\email{jllibre@uab.es}
\address{Adrian C. Murza, Institute of Mathematics ``Simion Stoilow''
of the Romanian Academy, Calea Grivi\c tei 21, 010702 Bucharest, Romania}
\email{adrian\_murza@hotmail.com}



\begin{abstract}
This is a survey on the Darboux theory of integrability for
polynomial vector fields, first in $\R^n$ and second in the
$n$-dimensional sphere $\sss^n$. We also provide new results about
the maximum number of parallels and meridians that a polynomial
vector field $\X$ on $\sss^n$ can have in function of its degree.
These results in some sense extend the known result on the maximum
number of hyperplanes that a polynomial vector field $\Y$ in $\R^n$
can have in function of the degree of $\Y$.
\end{abstract}

\maketitle

\section{Introduction and statement of the main results}

The real nonlinear ordinary differential systems are used to model a
wide range of processes practically all fields of science, from
biology and chemistry to economy, physics and engineering. The
existence of a first integrals of differential systems defined on
$\R^n$ is important for two main things. First, they make easier the
characterization the phase portrait of the system. Secondly, their
existence allow reducing the dimension of the system by one, which
in many cases makes easier the analysis. In our terminology, a
system is {\it integrable} if it has a first integral. Therefore,
the methods to detect the presence of first--integrals and find
their explicit form are extremely important in the qualitative
theory of differential equations.

However, as in many cases occurs, the more important the problem is,
the more difficult are the ways to access it. The techniques to
finding the presence and constructing first--integrals goes back to
Darboux, as far as we know \cite{Da}. Hamiltonians are the vector
fields whose first--integral are easiest to find. If the integrable
vector fields are not Hamiltonian, various techniques have been
developed for analyzing the existence of first integrals, such as
the Noether symmetries \cite{CS}, the already mentioned Darbouxian
theory of integrability \cite{Da} or the Lie symmetries \cite{Ol}.
In fact, Emmy Noether's Theorems represent a relevant example of the
interdisciplinary character acquired by the problem of finding first
integrals. By (roughly speaking), stating that any physical
conservation law has its associated symmetry, it establishes a
connection between mechanics, Lie algebra and differential
equations. Other contributions to this problem are represented by
the Painlev\'e analysis \cite{BGR}, the use of Lax pairs \cite{Lax}
or the direct method \cite{GRZ} and \cite{Hi}, to cite only a few of
them.

We are especially interested in the Darboux theory of integrability
for real polynomial vector fields. This theory provides a method of
constructing first--integrals of polynomial vector fields, based on
the number of invariant algebraic hypersurfaces that they have.
Since its publication in 1878, the method originally developed by
Darboux, has been extended and/or refined by many authors both in
$\R^2$ \cite{CLS, Ch, CL1, CL2, Da, Jo, Ll,
 LZ3, Pe, Po, PS, Sc1, Sc2, Sc3, Si}, and $\R^n$ see
\cite{LB, LM, LZ, LZ1, LZ2, LZ4}.

The first objective of this paper is to present the Darboux theory
of integrability of the real polynomial vector fields on the
$n$-dimensional sphere $\sss^n= \{x\in \R^2 : ||x||=1\}$ (here
$||\cdot ||$ denotes the Euclidean norm of $\R^{n+1}$), and to study
its relation with the maximum number of invariant parallels and
meridians that such vector fields can have. Our second objective is
to show the kind of parallelism that exist between the Darboux
theory of integrability of the real polynomial vector fields on
$\R^n$ and on $\sss^n$, and between the invariant hyperplanes of the
polynomial vector fields in $\R^n$ with the invariant parallels and
meridians of the polynomial vector fields on $\sss^n$.

\subsection{Darboux theory of integrability in $\R^n$ and invariant
hyperplanes}

We consider the following polynomial vector field
\[
\X=\displaystyle{\sum\limits_{i=1}\limits^{n+1}P_i(x_1,\ldots,x_{n+1})\frac{\p}{\p
x_i},\qquad (x_1,\ldots,x_{n+1})\in \mathbb R^{n+1}},
\]
in $\R^{n+1}$, where $P_i$ for $i=1,\ldots,n+1$ are polynomials of
degree at most $m$. If $m_i$ denotes the degree of the polynomial $P_i$,
we say that ${\bf m}=(m_1,\ldots,m_d)$ is the {\it degree} of the
polynomial vector field $\X$. Without loss of generality in the rest
of the paper we assume that $m_1\geqslant\ldots\geqslant m_d.$

Let $\C[x_1,\ldots,x_d]$ be the set of all polynomials in the
variables $x_1,\ldots,x_n$ with complex coefficients; in a similar
way we define $\R[x_1,\ldots,x_d]$. By an {\it invariant algebraic
hypersurface} of the polynomial vector field $\X$ we understand an
algebraic hypersurface $f=f(x_1,\ldots,x_d)=0$ with $f\in
\C[x_1,\ldots,x_d]$ such that for some polynomial $K\in
\C[x_1,\ldots,x_d],$ we have $\X f=Kf$. Moreover, this polynomial
$K$ is called the {\it cofactor} of the invariant algebraic
hypersurface $f=0.$ It is worth to noticing that if the degree of
the polynomial vector field $\X$ is ${\bf m}=(m_1,\ldots,m_d),$ with
$m_1\ge m_2 \ge \ldots \ge m_d$, then any cofactor has at most
degree $m_1-1$. Moreover, the algebraic hypersurface $f=0$ is called
an {\it invariant hyperplane} if the degree of $f$ is $1$.

The invariant algebraic hypersurfaces are the "building blocks" for
the Darbouxian first--integrals. Their property to separate the
phase space of the polynomial vector field $\X$ into invariant
pieces makes easier the study of its dynamics.

While $\X$ is a real polynomial vector field, we are working with
its complex invariant algebraic hypersurfaces, which in some cases
may be real. After stating the next theorem we will rigourously
explain this particular situation. For brevity, we remind here that
in some cases finding the real first integrals of the real
polynomial vector field $\X$ may require the explicit form of the
complex invariant algebraic hypersurfaces.

The invariance property of $f$ under the vector field, expressed by
$\X f=Kf,$ leads to the follow remark. Suppose that an orbit of the
polynomial vector field $\X$ intersects at a point the algebraic
hypersurface $f=0$. In that case, the entire orbit is contained in
$f=0$. In fact we call $f=0$ {\it flow--invariant}, because it is
invariant by $\X$. The converse is also true, see \cite{LM}.

One of the important benefits from refining the Darboux
integrability theory concerns the exponential factors. An {\it
exponential factor} $F(x_1,\ldots,x_n)$ of the polynomial vector
field $\X$ of degree $m$ is defined as an exponential function of
the form $\exp(g/h)$ with $g$ and $h$ polynomials in
$\C[x_1,\ldots,x_{n+1}]$ satisfying $\X F=LF$ for some $L\in
\C_{m_1-1}[x_1,\ldots,x_n]$. Here $\C_{m_1-1}[x_1,\ldots,x_n]$
denotes the set of all polynomials of $\C[x_1,\ldots,x_n]$ of degree
at most $m_1-1$. The exponential factors were introduced by
Christopher \cite{Ch}. They are strongly related to the multiplicity
of the invariant algebraic hypersurface $h=0$ if $h$ is not a
constant; if $h$ is constant it is related to the multiplicity of
the infinity, as discussed in \cite{CLP, LZ1, LZ4}. The fact that an
invariant algebraic hypersurface $h=0$ has multiplicity $k$ for a
polynomial vector field $\X,$ admits the following interpretation.
By doing a small perturbation to the vector field $\X$ within the
class of polynomial vector fields of the same degree, there are
polynomial vector fields $\Y$ near $\X$ having $k$ different
invariant algebraic hypersurfaces such that when $\Y \to \X$ these
$k$ hypersurfaces approach to the hypersurface $h=0$.

Let $U\subset \R^{n}$ be an open set. We say that a real function
$H(x_1,\ldots,x_{n},t): \R^{n} \times \R \to \mathbb R$, is an {\it
invariant} of the polynomial vector field $\X$ on $U$, if
$H(x_1(t),\ldots,x_{n}(t),t)$ is constant for all the values of $t$
for which the orbit $(x_1(t),\ldots,x_{n}(t))$ of $\X$ is contained
in $U$.

If the invariant function $H$ is independent on $t$, it is a {\it first
integral}. If a first integral is a rational function in its
variables, then it is called a {\it rational first integral} of the
vector field $\X$ on $U$.

If the invariant function $H$ is independent on $t$, it is a {\it first
integral}. We say that a {\it rational first integral} of the
vector field $\X$ on $U$ if it is a rational function in its
variables.

In the next theorem we summarize the Darboux theory for polynomial
vector fields in $\R^n$ which has been extended to the Darboux
theory for polynomial vector fields in $\sss^n$.

\begin{theorem}\label{t1}
Suppose that a polynomial vector field $\X$ defined in $\mathbb R^n$
of degree $m= (m_1,\ldots,m_n)$ admits $p$ invariant algebraic
hypersurfaces $f_i=0$ with cofactors $K_i$ for $i=1,\ldots,p$, and
$q$ exponential factors $F_j=\exp(g_j/h_j)$ with cofactors $L_j$ for
$j=1,\ldots,q$. Then the following statements hold.
\begin{itemize}
\item[$(a)$] There exist $\lambda_i,\mu_j\in \mathbb C$ not all
zero such that $\displaystyle{\sum\limits_{i=1}\limits^{p}\lambda_iK_i+
\sum\limits_{j=1}\limits^{q} \mu_jL_j=0}$, if and only if the real
$($multi-valued$)$ function of Darbouxian type
\begin{equation*}\label{e1}
f_1^{\lambda_1}\ldots f_p^{\lambda_p}F_1^{\mu_1}\ldots F_q^{\mu_q},
\end{equation*}
substituting $f_i^{\lambda_i}$ by $\left|f_i\right|^{\lambda_i}$ if
$\lambda_i\in \mathbb R$, is a first integral of the vector field
$\X$.

\smallskip

\item[$(b)$] If $p+q\geqslant \left(\begin{array}{c}n+m_1-1\\m_1-1\end{array}
\right)+1$, then there exist $\lambda_i,\mu_j\in \mathbb C$ not all
zero such that $\displaystyle{\sum\limits_{i=1}\limits^{p}\lambda_iK_i+
\sum\limits_{j=1} \limits^{q}\mu_jL_j=0}$.

\smallskip

\item[$(c)$] There exist $\lambda_i,\mu_j\in \mathbb C$ not all zero
such that $\displaystyle{\sum\limits_{i=1}\limits^{p}\lambda_iK_i+
\sum\limits_{j=1} \limits^{q}\mu_jL_j=-\sigma}$ for some $\sigma\in
\mathbb R\backslash\{0\}$, if and only if the real
$($multi-valued$)$ function
\[
f_1^{\lambda_1}\ldots f_p^{\lambda_p}F_1^{\mu_1}\ldots
F_q^{\mu_q}e^{\sigma t},
\]
substituting $f_i^{\lambda_i}$ by $\left|f_i\right|^{\lambda_i}$ if
$\lambda_i\in\mathbb R$, is an invariant of the vector field $\X$.

\smallskip

\item[$(d)$] The vector field $\X$ has a rational first integral if
and only if
\[
p+q\geqslant\left(\begin{array}{c}n+m_1-1\\m_1-1\end{array}\right)+n.
\]
Moreover, all trajectories are contained in invariant algebraic
hypersurfaces.
\end{itemize}
\end{theorem}

For a proof of Theorem \ref{t1} see for instance \cite{DLA, LZ1,
LZ2, LZ3}.

In the statement of the Darboux theory of integrability for
polynomial vector fields in $\R^n$ given in Theorem \ref{t1} we only
have taken into account partially the multiplicity of the invariant
algebraic hypersurfaces $f_i=0$ through the existence of integrating
factors, for more details see \cite{CLP} and \cite{LZ1}.

In statement (a) of Theorem \ref{t1} we said that the function
\eqref{e1} is real. This is due to the following fact. Since the
vector field $\X$ is real, it is well known that if a complex
invariant algebraic hypersurface or exponential factor appears, then
its conjugate has to appear simultaneously. If among the invariant
algebraic hypersurfaces of $\X$ a complex conjugate pair $f=0$ and
$\overline f=0$ occur, then the first integral \eqref{e1} has a real
factor of the form $f^{\lambda}{\bar f}^{\,\bar\lambda}$, which is
the multi-valued real function
\[
\left[\left(\mbox{Re}\,f\right)^2+\left(\mbox{Im}\,
f\right)^2\right]^{\mbox{Re}\,
\lambda}\,\mbox{exp}\left(-2\,\mbox{Im}\,\lambda\,\mbox{arctan}\,
\left(\frac{\mbox{Im}\,f}{\mbox{Re}\, f}\right)\right),
\]
if $\mbox{Im}\,\lambda\,\mbox{Im}\,f\not\equiv 0$. If among the
exponential factors of $\X$ a complex conjugate pair
$F=\mbox{exp}(h/g)$ and $\overline F=\mbox{exp}(\overline
h/\overline g)$ occur, the first integral \eqref{e1} has a real
factor of the form
\[
\left(\mbox{exp}\left(\frac{h}{g}\right)\right)^{\mu}\left(\mbox{exp}
\left(\frac{\overline h}{\overline g}\right)\right)^{\overline
\mu}=\mbox{exp}\left(2\,\mbox{Re}\,\left(\mu\,\frac{h}{g}\right)\right).
\]

One of the best tools for searching for invariant algebraic
hypersurfaces is the {\it extactic polynomial of $\X$ associated to
$W$}. To our knowledge it was first mentioned in the work of
Lagutinskii, see \cite{SW}. To define it let $W$ be a finitely
generated vector subspace of vector space $\C[x_1,\ldots,x_d]$. The
{\it extactic polynomial of $\X$ associated to $W$} is
\begin{equation*}\label{21}
{\mathcal E}_W={\mathcal E}_{\{v_{1}, \ldots, v_{l}\}}(\X) =
\det\left(
\begin{array}{cccc}
v_{1} & v_{2} & \ldots & v_{l} \\
\X(v_{1}) & \X(v_{2}) & \ldots & \X(v_{l}) \\
\vdots & \vdots & \ldots & \vdots \\
\X^{l-1}(v_{1}) & \X^{l-1}(v_{2}) & \ldots & \X^{l-1}(v_{l})
\end{array}
\right)=0,
\end{equation*}
where $\{v_{1}$, \ldots, $v_{l}\}$ is a basis of $W$, $l$ is the
dimension of $W$, and $\X^{j}(v_{i})= \X^{j-1}(\X(v_{i}))$. It is
well known that one of the main properties of the extactic
polynomial is that its definition does not dependent of the chosen
basis of $W.$

In this paper the reason to using the extactic polynomial ${\mathcal
E}_{W}(\X)$ is twofold. Firstly, it allows detecting algebraic
hypersurfaces $f=0$ with $f\in W$ such that they are invariant by
the polynomial vector field $\X$, see the following proposition
proved in \cite{CLP}. Secondly, it allows to defining and computing
the multiplicity of invariant algebraic hypersurfaces.

Even if the next proposition is stated for complex polynomial vector
fields, it is very useful for our later considerations. This is so,
because we deal with real polynomial vector fields, which are
particular cases of complex ones.

\begin{proposition}\label{p1}
Let $\X$ be a polynomial vector field in $\C^d$ and let $W$ be a
finitely generated vector subspace of $\C[x_1,\ldots,x_d]$ with
$\dim (W)>1 $. Then every algebraic invariant hypersurface $f=0$ for
the vector field $\X$, with $f\in W$, is a factor of the polynomial
${\mathcal E}_{W}(\X)$.
\end{proposition}

{From} Proposition \ref{p1} it follows that $f=0$is an invariant
hyperplane of the polynomial vector field $\X$ if the polynomial $f$
is a factor of the polynomial ${\mathcal E}_{W}(\X)$, where $W$ is
generated by $\{1,x_1,\ldots,x_d\}$. From \cite{CLP} {\it the
invariant hyperplane $f=0$ has multiplicity $k$} if $k$ is the
greatest positive integer such that $f^k$ divides the polynomial
${\mathcal E}_{W}(\X)$.

An important result on the maximum number of invariant hyperplanes
that a polynomial vector field in $\R^n$ can have was proved in
\cite{LM}. The analogous result for polynomial vector fields in
$\R^2$ was proved before in \cite{AGL}.

\begin{theorem}\label{t2}
Assume that a polynomial vector field $\X$ in $\R^n$ or $\C^n$ with
$n\ge 2$ of degree ${\bf m}= (m_1,\ldots,m_d)$ with $m_1\ge \ldots
\ge m_n$ has finitely many invariant hyperplanes. Then
\begin{enumerate}
\item{}\label{part1} the number of
invariant hyperplanes of $\X$ taking into account their
multiplicities is at most
\begin{equation*}\label{3}
\displaystyle{\left(\begin{array}{c} n\\
2 \end{array} \right)(m_1-1)+\left( \sum_{k=1}^n m_k\right)};
\end{equation*}
\item{}\label{part2} the number of
invariant hyperplanes of $\X$ through a single point taking into account their
multiplicities taking into account their
multiplicities is at most
\begin{equation*}\label{31}
\displaystyle{\left(\begin{array}{c} n-1\\
2 \end{array} \right)(m_1-1)+\left( \sum_{k=1}^{n-1} m_k\right)+1}.
\end{equation*}
\end{enumerate}
Moreover there are complex polynomial vector fields which reach this
upper bound for their number of invariant hyperplanes taking into
account their multiplicities.
\end{theorem}

\begin{remark}
Knowing a sufficient number of invariant hyperplanes such that
their cofactors satisfy the assumptions either of statement (a) or
of statement (b) of Theorem \ref{t1} enable us to find a first integral of
the polynomial vector field $\X$.
\end{remark}

\subsection{Darboux theory of integrability in $\sss^n$ and invariant
parallels and meridians}

Before stating our results, we need a few preliminary definitions
for hypersurfaces we are dealing with. First we give the general
introductory notions valid for the application of the Darboux theory
of integrability on any smooth hypersurfaces; then we will focus on
the $n$--dimensional sphere $\sss^n$ and its invariant meridians and
parallels.

Let $G:\R^{n+1}\to\R$ be a $C^1$ map. A hypersurface
${\bf\Omega}=\{(x_1,\ldots,x_{n+1})$ $\in \R^{n+1}:
G(x_1,\ldots,x_{n+1})=0\}$ is said to be {\it regular} if the
gradient $\nabla G$ of $G$ is not equal to zero on $\bf \Omega$. Of
course, if $\Omega$ is regular, then it is smooth. We say that
$\Omega$ is {\it algebraic} if $G$ is an irreducible polynomial. If
the degree of the polynomial $G$ is $d$, then we say that $\Omega$
is algebraic of {\it degree} $d$. A {\it polynomial vector field}
$\X$ {\it on the regular hypersurface} $\Omega$ (or simply a {\it
polynomial vector field on} $\Omega$) is a polynomial vector field
$\X$ in $\R^{n+1}$ satisfying
\begin{equation}\label{on}
(P_1,\ldots,P_{n+1})\cdot\nabla G=0\,\, \mbox{ on the points of
}\Omega,
\end{equation}
where the dot denotes the inner product of two vectors in
$\R^{n+1}$. If the polynomial vector field $\X$ in $\R^{n+1}$ has
degree $m$, then we say that the vector field $\X$ on $\Omega$ is of
degree $m$.

Let $f=f(x_1,\ldots,x_{n+1})\in \mathbb C[x_1,\ldots,x_{n+1}]$. We
say that the algebraic hypersurface
$\{f=0\}\cap\Omega\subset\R^{n+1}$ is {\it invariant by the
polynomial vector field} $\X$ {\it on} $\Omega$ (or simply an {\it
invariant algebraic hypersurface on } $\Omega$) if it satisfies
\begin{itemize}
\item[(i)] there exists a polynomial $k\in\mathbb C[x_1,\ldots,x_{n+1}]$
such that
\[
\displaystyle{\X f=\sum\limits_{i=1}\limits^{n+1}P_i\frac{\p f}{\p x_{i}}=kf\,\,
\mbox{ on }\Omega},
\]
the polynomial $k=k(x_1,\ldots,x_{n+1})\in \mathbb
C[x_1,\ldots,x_{n+1}]$ is called the {\it cofactor} of $f=0$ on
$\Omega$;

\smallskip

\item[(ii)] the two hypersurfaces $f=0$ and $\Omega$ have transversal
intersection, i.e. if the vectors $\nabla G$ and $\nabla f$ are
independent in all the points of the hypersurface
$\{f=0\}\cap\Omega$.
\end{itemize}

Clearly the vector field $\X$ is tangent to the algebraic
hypersurface $\{f=0\}\cap\Omega$. So the hypersurface
$\{f=0\}\cap\Omega$ is formed by orbits of the vector field $\X$.
This explains why we say that the algebraic hypersurface
$\{f=0\}\cap\Omega$ is invariant by the flow of the vector field
$\X$.

An {\it exponential factor} $F=F(x_1,\ldots,x_{n+1})$ of the
polynomial vector field $\X$ of degree $m$ on the regular
hypersurface $\Omega$ is an exponential function of the form
$\exp(g/h)$ with $g$ and $h$ polynomials in $\C[x_1,\ldots,x_{n+1}]$
and satisfying $\X F=LF$ on $\Omega$ for some $K\in
\C_{m-1}[x_1,\ldots,x_{n+1}]$.

Let $f$ and $g$ be two polynomials of $\C_m[x_1,x_2,\ldots,x_{n+1}]$
and let $\Omega=\{G=0\}$ be a smooth algebraic hypersurface in
$\R^{n+1}$ of degree $d$. We say that $f$ and $g$ are related,
$f\sim g$, if either $f/g= \mbox{constant}$ or $f-g=hG$ for some
polynomial $h$. That is, $\sim$ is an equivalence relation in
$\C_m[x_1,x_2,\ldots,x_{n+1}];$ it partitions
$\C_m[x_1,x_2,\ldots,x_{n+1}]$ into equivalence classes defined in
the following way. Given a set $\C_m[x_1,x_2,\ldots,x_{n+1}]$ and
the equivalence relation $\sim$ on $\C_m[x_1,x_2,\ldots,x_{n+1}],$
the equivalence class of an element $g$ in
$\C_m[x_1,x_2,\ldots,x_{n+1}]$ is the set
$$\displaystyle{\left\{f\in\C_m[x_1,x_2,\ldots,x_{n+1}]|f\sim g\right\}}.$$
In both cases mentioned in this paragraph $\Omega\cap \{f=0\}=
\Omega\cap \{g=0\}$. The result of the partition of
$\C_m[x_1,x_2,\ldots,x_{n+1}]$ by the equivalence relation $\sim$
into equivalence classes yields the quotient space
$\C_m[x_1,\ldots,x_{n+1}]/\sim;$ we denote its dimension by $d(m)$,
called the {\it dimension of} $\C_m[x_1,\ldots,x_{n+1}]$ {\it on}
$\Omega$. In \cite{LZ} it is proved that the dimension of $\mathbb
C_m[x_1,\ldots,x_n]/\sim$ is
\[
d(m)=\left(\begin{array}{c}n+m\\n\end{array}\right)-
\left(\begin{array}{c}n+m-d\\n\end{array}\right).
\]

Let $U\subset \R^{n+1}$ be an open set. A real function
$H(x_1,\ldots,x_{n+1},t): \R^{n+1} \times \R \to \mathbb R$, is said
to be {\it invariant} of the polynomial vector field $\X$ on
$\Omega\cap U$, if $H(x_1(t),\ldots,x_{n+1}(t),$
$t)=\mbox{constant}$ for all the values of $t$ for which the orbit
$(x_1(t),\ldots,x_{n+1}(t))$ of $\X$ is contained in $\Omega\cap U$.

If an invariant $H$ is independent on $t$, then $H$ is a {\it first
integral}. If a first integral $H$ is a rational function in its
variables, then it is called a {\it rational first integral} of the
vector field $\X$ on $\Omega\cap U$.

Now we present the extension of the Darboux theory of integrability
to polynomial vector fields on $\sss^n$.

\begin{theorem}\label{t3}
Assume that $\X$ is a polynomial vector field on $\sss^n$ of degree
$m=(m_1,\ldots,m_n)$ having $p$ invariant algebraic hypersurfaces
$\{f_i=0\}\cap\sss^n$ with cofactors $K_i$ for $i=1,\cdots,p$, and
$q$ exponential factors $F_j=\exp(g_j/h_j)$ with cofactors $L_j$ for
$j=1,\cdots,q$. Then the following statements hold.
\begin{itemize}
\item[(a)] There exist $\lambda_i,\mu_j\in \mathbb C$ not all
zero such that $\displaystyle{\sum\limits_{i=1}\limits^{p}\lambda_iK_i+
\sum\limits_{j=1}\limits^{q} \mu_jL_j=0}$ on $\sss^n$,  if and only
if the real $($multi--valued$)$ function of Darbouxian type
\begin{equation*}
f_1^{\lambda_1}\cdots f_p^{\lambda_p}F_1^{\mu_1}\cdots F_q^{\mu_q},
\label{a1}
\end{equation*}
substituting $f_i^{\lambda_i}$ by $\left|f_i\right|^{\lambda_i}$ if
$\lambda_i\in \mathbb R$, is a first integral of the vector field
$\X$ on $\sss^n$.

\smallskip

\item[(b)] If $p+q\geqslant \frac{\displaystyle n+2m_1}{\displaystyle
n+m_1}\left(\begin{array}{c}n+m_1\\m_1\end{array}\right)+1$, then
there exist $\lambda_i,\mu_j \in \mathbb C$ not all zero such that
$\displaystyle{\sum\limits_{i=1}\limits^{p}
\lambda_iK_i+\sum\limits_{j=1}\limits^{q}\mu_jL_j=0}$ on $\sss^n$.

\smallskip

\item[(c)] There exist $\lambda_i,\mu_j\in \mathbb C$ not all zero
such that $\displaystyle{\sum\limits_{i=1}\limits^{p}\lambda_iK_i+
\sum\limits_{j=1}\limits^{q} \mu_jL_j=-\sigma}$ on $\sss^n$ for some
$\sigma\in\mathbb R\backslash\{0\}$, if and only if the real
$($multi--valued$)$ function
\[
f_1^{\lambda_1}\cdots f_p^{\lambda_p}F_1^{\mu_1}\cdots
F_q^{\mu_q}e^{\sigma t},
\]
substituting $f_i^{\lambda_i}$ by $\left|f_i\right|^{\lambda_i}$ if
$\lambda_i\in\mathbb R$, is an invariant of $\X$ on $\sss^n$.

\smallskip

\item[(d)] The vector field $\X$ on $\sss^n$ has a rational first
integral if and only if $p+q\geqslant\frac{\displaystyle
n+2m_1}{\displaystyle n+m_1}\left(\begin{array}{c}n+
m_1\\m_1\end{array}\right)+n$. Moreover, all trajectories are
contained in invariant algebraic hypersurfaces.
\end{itemize}
\end{theorem}

The proof of statements (a), (b) and (c) of Theorem \ref{t3} was
done in \cite{LZ}, and the proof of its statement (d) in \cite{LB}.

The {\it parallels} of the $n$-dimensional sphere $\sss^n$ are the
intersections of the hyperplanes $x_{n+1}=$ constant with the sphere
$\sss^n$, and the {\it meridians} are the intersections of the
hyperplanes containing the $x_{n+1}$--axis with the sphere $\sss^n$.
Now we need to extend Theorem \ref{t2} to the parallels and
meridians.

In the next theorem we provide the maximum number of invariant
meridians that a polynomial vector field in $\sss^n$ can have in
function of its degree.

\begin{theorem}\label{t4}
For $n\geqslant 2$ let $\mathcal{X}$ be a polynomial vector field on
$\sss^n$ of degree ${\bf m}=(m_1,\ldots,m_{n+1})$ with $m_1\geqslant
m_2\geqslant\ldots\geqslant m_{n+1}.$ Assume that $\mathcal{X}$ has
finitely many invariant meridians. Then the number of invariant
meridians of $\mathcal{X}$ is at most
\begin{equation}\label{mer}
\displaystyle \binom{n-1}{2}
(m_1-1)+ \left(\sum_{i=1}^{n-1} m_i\right)+1,
\end{equation}
 where $\displaystyle
\binom{n-1}{2}=0$ if $n=2$.
\end{theorem}

\begin{proof}
By definition an invariant meridian of $\sss^n$ by the vector field
$\X$ is obtained intersecting an invariant hyperplane of $\X$ of the
form $\displaystyle \sum_{i=1}^n a_ix_i=0$ (which contain the
$x_{n+1}$--axis) with $\sss^n$. So in order to determine an upper
bound for the number of invariant meridians of $\X$ it suffices to
determine an upper bound for the number of invariant hyperplanes
$\displaystyle \sum_{i=1}^n a_ix_i=0$ of $\X$.

By Proposition \ref{p1} if  $\displaystyle \sum_{i=1}^n a_ix_i=0$ is
an invariant hyperplane of $\X$ it must divide the extactic
polynomial $\mathcal{E}_{\{x_1\ldots,x_n\}}(\X)$. Therefore the
degree of this extactic polynomial is an upper bound for the maximal
number of meridians.

{From} the definition of extactic polynomial we have
\begin{equation*}\label{mer12}
\mathcal{E}_{\{x_1,\ldots,x_n\}}(\X) = \left|
\begin{array}{cccccc}
x_{1} & x_{2} & x_{3} & \cdots & x_{n-1} & x_n \\
P_1 & P_2 & P_3 & \cdots & P_{n-1} & P_n\\
\X(P_1) & \X(P_2) &\X(P_3) & \cdots & \X(P_{n-1}) & \X(P_n)\\
\vdots & \vdots & \vdots & \vdots & \vdots \\
\X^{n-4}(P_1) & \X^{n-4}(P_2) & \X^{n-4}(P_3) & \cdots & \X^{n-4}(P_{n-1})& \X^{n-4}(P_n)\\
\X^{n-3}(P_1) & \X^{n-3}(P_2) & \X^{n-3}(P_3) & \cdots & \X^{n-3}(P_{n-1})& \X^{n-3}(P_n)\\
\X^{n-2}(P_1) & \X^{n-2}(P_2) & \X^{n-2}(P_3) & \cdots &
\X^{n-2}(P_{n-1})& \X^{n-2}(P_n)
\end{array}
\right|
\end{equation*}
Since $m_1\geqslant m_2\geqslant m_3\geqslant\ldots\geqslant
m_{n+1}$, we have that the degree of the extactic polynomial
$\mathcal{E}_{\{x_1,\ldots,x_n\}}(\X)$ is at most the degree of the
polynomial
\[
\mathcal P=\X^{n-2}(P_1) \X^{n-3}(P_2) \X^{n-4}(P_3) \cdots
\X(P_{n-2}) P_{n-1} x_n.
\]
Clearly the degree of $\X^\ell(P_i)$ is $\ell(m_1-1)+m_i$, so the
degree of $\mathcal P$ is
\[
\begin{array}{rl}
\deg(\mathcal P)=& \big((n-2)(m_1-1)+m_1\big)+
\big((n-3)(m_1-1)+m_2\big) \\
& + \cdots + \big((m_1-1)+m_{n-2}\big)+m_{n-1}+1.
\end{array}
\]
If $n>2$ then
\[
\begin{array}{rl}
\deg(\mathcal P=& \displaystyle\left( \sum_{k=1}^{n-2} k\right)
(m_1-1)+ \left(
\sum_{k=1}^{n-1} m_k\right)+1\vspace{0.2cm}\\
 =& \displaystyle
\binom{n-1}{2}(m_1-1)+ \left(\sum_{i=1}^{n-1}m_i\right)+ 1.
\end{array}
\]
If $n=2$ then
\[
\deg(\mathcal P)=m_1+1.
\]
This completes the proof of the theorem.
\end{proof}

In the next theorem we provide the maximum number of invariant
parallels that a polynomial vector field in $\sss^n$ can have in
function of its degree.

\begin{theorem}\label{t5}
Let $\mathcal{X}$ be a polynomial vector field on $\sss^n$ of degree
${\bf m}=(m_1,\ldots,m_{n+1})$ with $m_1\geqslant m_2\geqslant
\ldots\geqslant m_{n+1}.$ Assume that $\mathcal{X}$ has finitely
many invariant parallels. Then the number of invariant parallels of
$\mathcal{X}$ is at most $m_{n+1}$.
\end{theorem}

\begin{proof}
By definition an invariant parallel is the intersection of an
invariant hyperplane of the form $x_{n+1}=k$, where $k\in (-1,1)$,
with the sphere $\sss^n$. Thus this intersection is an $\sss^{n-1}$
sphere of radius $\sqrt{1-k^2}$.

{From} Proposition \ref{p1} we know that if $x_{n+1}-k=0$ is an
invariant hyperplane of polynomial vector field $\X$, then
$x_{n+1}-k$ is a factor of the extactic polynomial
\begin{equation*}
\mathcal{E}_{\{1,x_{n+1}\}}(\mathcal{X})=\left| \begin{array}{cc}1&x_{n+1}\\
\mathcal{X}(1)&\mathcal{X}(x_{n+1}) \end{array} \right|=
\left| \begin{array}{cc}1&x_{n+1}\\
0&\dot{x}_{n+1} \end{array}
\right|=\dot{x}_{n+1}=P_{n+1}(x_1,\ldots,x_{n+1}).
\end{equation*}
Since the degree of $P_{n+1}$ is $m_{n+1}$, this polynomial at most
can have $m_{n+1}$ linear factors of the form $x_{n+1}-k$. Hence it
follows that the number of invariant parallels of $\mathcal{X}$ is
at most $m_{n+1}$.
\end{proof}

Next we will prove that if the polynomial vector field $\X$ on
$\sss^2$ of degree ${\bf m}=(m_1,m_2)$ with $m_1\geqslant m_2$ is
complex, then it can reach the maximum upper bound $m_1+1$ for the
number of invariant meridians stated in Theorem \ref{t4}, but if
$\X$ is real then we have examples with $m_1$ invariant meridians.
This result that for real polynomial vector fields the upper bounds
of some invariant objects cannot be reached, but they can be reached
for the complex polynomial vector fields already occurred for
polynomial vector field $\X$ in the plane. Thus in the complex plane
$\C^2$ a polynomial vector field $\X$ of degree ${\bf m}=(m_1,m_2)$
can have at most $3m_1-1$ invariant straight lines, and this bound
is reached, but if $\X$ is in $\R^2$ we only know that this bound is
in the interval $[2m_1+1,3m_1-1]$ if $m_1$ is even and
$[2m_1+2,3m_1-1]$ if $m_1$ is odd, see for more details \cite{AGL,
CLP}.

In the following proposition we provide a complex polynomial vector
field $\X$ of degree $(2,2)$ on the complex sphere $\sss^2$ where
the upper bound $m_1+1=3$ for the maximum number of invariant
meridians is reached, and we prove that the real polynomial vector
field $\X$ of degree ${\bf m}=(2,2)$ on the real sphere $\sss^2$
cannot have $m_1+1=3$ invariant meridians. Additionally we provide a
real polynomial vector field $\X$ of degree $(2,2)$ on the real
sphere $\sss^2$ with $m_1=2$ invariant meridians.

\begin{proposition}\label{p9}
The following statements hold.
\begin{itemize}
\item[(a)] The complex polynomial differential system
\[
\dot x= i y (x+y)-2 x z, \quad \dot y= -i x (x+y)-2 y z,\quad \dot
z= 1 + x^2 + y^2 - z^2,
\]
on the complex sphere $\sss^2$ of degree $(2,2,2)$ has the maximum
number of complex invariant meridians $m_1+1=3$, namely $x+i y=0$,
$x-iy=0$ and $x+y+1=0$.

\item[(b)] There are no real polynomial differential systems on the
real sphere $\sss^2$ of degree $(2,2,2)$ realizing the maximum
number of real invariant meridians $m_1+1=3$. Such polynomial
differential systems have at most $2$ real invariant meridians, and
this upper bound is reached.
\end{itemize}
\end{proposition}

\begin{proof}
Consider the vector field $\X$ associated to the complex polynomial
differential system of statement $(a)$. Then
\[
\X(x^2+y^2+z^2-1)= -2 z (x^2 + y^2 + z^2-1),
\]
so $\X$ is a polynomial vector field on the complex sphere $\sss^2$.
The three complex planes $x+i y=0$, $x-iy=0$ and $x+y+1=0$ are
invariant because
\begin{equation}\label{pp1}
\begin{array}{l}
\X(x+i y)=  (x + y - 2 z)(x + i y), \vspace{0.2cm}\\
\X(x-i y)=  -(x + y + 2 z)(x - i y), \vspace{0.2cm}\\
\X(x+y)= -i (x - y - 2 i z)(x + y).
\end{array}
\end{equation}
This completes the proof of statement $(a)$.

Now we shall prove statement $(b)$. Consider a general real
polynomial differential system of degree $(2,2,2)$ in $\R^3$, i.e.
\[
\dot x= \sum_{i+j+k=0}^2 a_{ijk} x^i y^j z^k, \quad \dot y=
\sum_{i+j+k=0}^2 b_{ijk} x^i y^j z^k, \quad \dot z= \sum_{i+j+k=0}^2
c_{ijk} x^i y^j z^k.
\]
Imposing that $\X(x^2+y^2+z^2-1)= (k_0+k_1 x+k_2
y+k_3z)(x^2+y^2+z^2-1)$ system \eqref{pp1} becomes
\begin{equation}\label{pp2}
\begin{array}{rl}
\dot x=& a_{000}  - b_{100} y - c_{100} z - a_{000} x^2- (b_{000} +
b_{200}) x y  - (c_{000} + c_{200}) x z\\
& - (a_{000} + b_{110}) y^2 - (b_{101} + c_{110}) y z -(a_{000} +
c_{101}) z^2, \vspace{0.2cm}\\
\dot y=& b_{000} + b_{100} x + b_{200} x^2 + b_{110} x y - b_{000}
y^2 - c_{010} z + b_{101} x z\\
&  - (c_{000} + c_{020}) y z - (b_{000} + c_{011}) z^2, \vspace{0.2cm}\\
\dot z=& c_{000} + c_{100} x + c_{200} x^2 + c_{010} y + c_{110} x y
+ c_{020} y^2 + c_{101} x z\\
& + c_{011} y z - c_{000} z^2.
\end{array}
\end{equation}

Now we impose that the plane $ax+by=0$ be invariant by the
polynomial differential system \eqref{pp2}, i.e. $\X(a+by)= (k_0+k_1
x+k_2 y+k_3z)(ax+by)$ and we obtain three possible polynomial
differential systems \eqref{pp2} having invariant planes of this
kind.

First the polynomial differential system
\begin{equation*}\label{pp3}
\begin{array}{rl}
\dot x=& a_{000} - c_{100}z- a_{000} x^2  - (c_{000} +c_{200}) x z
- (a_{000} + b_{110}) y^2 \\
&  -c_{110} y z - (a_{000} +c_{101}) z^2, \vspace{0.2cm}\\
\dot y=& b_{110} x y - (c_{000} +c_{020}) y z, \vspace{0.2cm}\\
\dot z=& c_{000} + c_{100} x + c_{200} x^2 + c_{110} x y + c_{020}
y^2 + c_{101} x z - c_{000} z^2,
\end{array}
\end{equation*}
having the invariant plane $y=0$, and at most it has another
invariant plane of the form $ax+by=0$, namely $x=0$, because its
extactic polynomial is
\[
y (b_{110} (x^2 + y^2) + z (c_{100} - c_{020} x + c_{200} x +
c_{110} y + c_{101} z) + a_{000} (-1 + x^2 + y^2 + z^2)),
\]
and taking $b_{110}=c_{100}=c_{110}=c_{101}=a_{000}=0$ it is easy to
check that it has also the invariant plane $x=0$.

Second the polynomial differential system
\begin{equation*}\label{pp4}
\begin{array}{rl}
\dot x=& -(b_{000} +b_{200}) x y - (c_{000} + c_{200}) x z, \vspace{0.2cm}\\
\dot y=& b_{000} - c_{010} z + b_{200} x^2 - b_{000} y^2 +
b_{101} x z - (c_{000} +c_{020}) y z\\
&  - (b_{000} +c_{011}) z^2, \vspace{0.2cm}\\
\dot z=& c_{000} + c_{010} y + c_{200} x^2 - b_{101} x y + c_{020}
y^2 + c_{011} y z - c_{000} z^2,
\end{array}
\end{equation*}
having the invariant plane $x=0$, and at most it has another
invariant plane of the form $ax+by=0$, namely $y=0$, because its
extactic polynomial is
\[
x (b_{000} + b_{200} (x^2 + y^2) + (-c_{010} + b_{101} x + (-c_{020}
+ c_{200}) y) z - (b_{000} + c_{011}) z^2);
\]
and taking $b_{000}=b_{200}=c_{010}=b_{101}=c_{011}=0$ it is easy to
check that it has also the invariant plane $y=0$.

Finally the polynomial differential system
\begin{equation*}\label{pp5}
\begin{array}{rl}
\dot x=& -b b_{000}/a - c_{100} z + b b_{000} x^2/a -(b_{000} +
b_{200})x y + (b b_{000}/a  -(c_{000} + c_{200})) x z\\
& - b (b_{000} + b_{200}) y^2/ a  -(b_{101} + c_{110}) y z +
(b b_{000}/a - c_{101}) z^2, \vspace{0.2cm}\\
\dot y=& b_{000}+ a c_{100} z/b + b_{200} x^2 + b (b_{000} +
b_{200}) x y/a  + b_{101} x z -b_{000} y^2\\
& + ( (a^2 b_{101} + b^2 b_{101} + a^2 c_{110} - a
b c_{200})/(a b)-c_{000}) y z + ( (a c_{101})/b-b_{000}) z^2, \vspace{0.2cm}\\
\dot z=& c_{000} + c_{100} x - a c_{100} y/b + c_{200} x^2 + c_{110}
x y+ c_{101} x z\\
& - (a^2 b_{101} + b^2 b_{101} + a^2 c_{110} - a b c_{200}) y^2/(a
b)  - a c_{101} y z/b - c_{000} z^2,
\end{array}
\end{equation*}
having the invariant plane $ax+by=0$ with $ab\ne 0$, and at most it
has another invariant plane of the form either $x=0$ or $y=0$,
because its extactic polynomial is
\[
((a x + b y) (a z (c_{100} + (b_{101} + c_{110}) y + c_{101} z) + b
(b_{000} + b_{200} (x^2 + y^2) + b_{101} x z - b_{000} z^2)))/(a b),
\]
taking either $c_{100}=b_{101} + c_{110}=c_{101}=b_{000}=b_{200}=0$,
or $c_{100}=c_{101}=b_{000}=b_{200}=b_{101}=0$, respectively.

In summary we have proved that quadratic polynomial differential
systems on the sphere $\sss^2$ can have at most $2$ invariant
meridians, and that there are such kind of systems having $2$
invariant meridians.
\end{proof}

In the next proposition we give a polynomial differential system on
the sphere $\sss^2$ where the upper bound for the maximum number of
invariant parallels is reached.

\begin{proposition}\label{p11}
Consider the polynomial differential system
\begin{equation}\label{ex2}
\begin{array}{l}
\dot{x}_2=1 - x - x^2 - y^2 + z^2, \vspace{0.2cm}\\
\dot{x}_3=-2 y z, \vspace{0.2cm}\\
\dot{x}_1= y,
\end{array}
\end{equation}
on the sphere $\sss^2$ of degree $(2,2,1)$. For this system the
upper bound $m_3=1$ for the number of invariant parallels provided
in Theorem \ref{t5} is reached.
\end{proposition}

\begin{proof}
Let $\X$ be the vector field associated to system \eqref{ex2}. Then
$$
\X(x^2+y^2+z^2-1)= -2y(x^2+y^2+z^2-1).
$$
So $\X$ is a vector field on the sphere $\sss^2$.  Clearly the plane
$z=0$ is invariant by the flow of $\X$, and is the unique invariant
plane of the form $z=$ constant. So the proposition is proved.
\end{proof}

\section*{Acknowledgements}

The first author is partially supported by a FEDER-MINECO grant
MTM2016-77278-P, a MINECO grant MTM2013-40998-P, and an AGAUR grant
2014SGR-568. The second author acknowledges a BITDEFENDER
postdoctoral fellowship from the Institute of Mathematics "Simion
Stoilow" of the Romanian Academy, Contract of Sponsorship No.
262/2016 as well as partial support from a grant of the Romanian
National Authority for Scientific Research and Innovation,
CNCS-UEFISCDI, project number PN-II-RU-TE-2014-4-0657.

\end{document}